\documentclass[10pt]{amsart}
\usepackage{amsmath,amscd}
\usepackage{amsbsy}
\usepackage{amssymb}
\usepackage{amscd,amsthm}

\usepackage[all,cmtip]{xy}

\newtheorem{thm}{Theorem}\numberwithin{thm}{section}
\newtheorem{lem}[thm]{Lemma}
\newtheorem{prop}[thm]{Proposition}
\newtheorem{cor}[thm]{Corollary}
\newtheorem{exam}[thm]{Example}
\newtheorem{rema}[thm]{Remark}

\newtheorem{defi}[thm]{Definition}

\newtheorem*{thm2}{Theorem}

\newtheorem*{con2}{Conjecture}

\begin{document}
\begin{center}
\huge{Tilting objects on some global quotient stacks}\\[1cm]
\end{center}
\begin{center}

\large{Sa$\mathrm{\check{s}}$a Novakovi$\mathrm{\acute{c}}$}\\[0,5cm]
\end{center}
{\small \textbf{Abstract}. 
We prove the existence of tilting objects on some global quotient stacks. As a consequence we provide further evidence for a conjecture on the Rouquier dimension of derived categories formulated by Orlov.
\begin{center}
\tableofcontents
\end{center}
\section{Introduction}
Geometric tilting theory started with the construction of tilting bundles on the projective space by Beilinson \cite{BE}. Later Kapranov \cite{KA}, \cite{KA1}, \cite{KA2} constructed tilting bundles for certain homogeneous spaces. Further examples can be obtained from certain blow ups and taking projective bundles \cite{CM}, \cite{CDR}, \cite{O}. A smooth projective $k$-scheme admitting a tilting object satisfies very strict conditions, namely its Grothendieck group is a free abelian group of finite rank and the Hodge diamond is concentrated on the diagonal, at least in characteristic zero \cite{BH}.

However, it is still an open problem to give a complete classification of smooth projective $k$-schemes admitting a tilting object. In the case of curves one can prove that a smooth projective algebraic curve has a tilting object if and only if the curve is a one-dimensional Brauer--Severi variety. Recall that a Brauer--Severi variety is a $k$-scheme becoming isomorphic to the projective space after base change to the algebraic closure $\bar{k}$. In this sense, one-dimensional Brauer--Severi varieties are very close to the projective line. But already for smooth projective algebraic surfaces there is currently no classification of surfaces admitting such a tilting object. It is conjectured that a smooth projective algebraic surface has a tilting bundle if and only if it is rational (see \cite{BSH}, \cite{H}, \cite{HP}, \cite{HP1}, \cite{HP2}, \cite{KI} and \cite{P} for results in this direction).

In the present work, we will focus on certain global quotient stacks and prove the existence of tilting objects for their derived category. Several examples of stacks admitting a tilting object are known (see \cite{IU}, \cite{IU1}, \cite{KAW}, \cite{ME}, \cite{OU} and \cite{OUE}). But as in the case of schemes, one has to settle for existence criteria for stacks admitting a tilting object. The first main result of the present paper is the following:
\begin{thm2}(Theorem 4.2)
Let $X$ be a smooth projective $k$-scheme and $G$ a finite group acting on $X$ with $\mathrm{char}(k) \nmid \mathrm{ord}(G)$. Suppose there is a $\mathcal{T}^{\bullet}\in D_G(\mathrm{Qcoh}(X))$ which, considered as an object in $D(\mathrm{Qcoh}(X))$, is a tilting object on $X$. Denote by $k[G]$ the regular representation of $G$, then $\mathcal{T}^{\bullet}\otimes k[G]$ is a tilting object on $[X/G]$.
\end{thm2}
This theorem enables us to find examples of quotient stacks $[X/G]$ admitting a tilting object. Notice that Elagin \cite{EL1} proved the existence of full strongly exceptional collections on $[X/G]$. Since there are $k$-schemes that have tilting objects, but not a full strongly exceptional collection (see Proposition 4.8), Theorem 4.2 indeed gives us some new examples (see Example 4.7). Moreover, exploiting the derived McKay correspondence, Theorem 4.2 also provides us with some crepant resolutions admitting a tilting object (see Corollaries 4.11, 4.12 and 4.13).

Next, we prove a result generalizing and harmonizing results of Bridgeland and Stern \cite{BS}, Theorem 3.6 (see also \cite{BRI}, Proposition 4.1) and Brav \cite{BR}, Theorem 4.2.1. For a finite group $G$ acting on $X$, let $\mathcal{E}$ be an equivariant locally free sheaf and $\mathbb{A}(\mathcal{E})$ the affine bundle. Suppose $\mathrm{char}(k) \nmid \mathrm{ord}(G)$ and denote by $\pi:\mathbb{A}(\mathcal{E})\rightarrow X$ the projection.  
\begin{thm2}(Theorem 5.1)
Let $X$ be a smooth projective $k$-scheme, $G$ a finite group acting on $X$ and $\mathcal{E}$ an equivariant locally free sheaf of finite rank. Suppose $\mathcal{T}$ is a tilting bundle on $[X/G]$. If $H^i(X,\mathcal{T}^{\vee}\otimes \mathcal{T}\otimes S^l(\mathcal{E}))=0$ for all $i\neq 0$ and all $l>0$, then $[\mathbb{A}(\mathcal{E})/G]$ admits a tilting bundle, too.
\end{thm2}
If $X$ is a Fano variety, $\mathcal{E}=\omega_X$ and $G=1$ we obtain the result in \cite{BS} and if $X=\mathrm{Spec}(\mathbb{C})$ the result in \cite{BR}. It is also natural to consider projective bundles of equivariant locally free sheaves $\mathcal{E}$ on $X$. We prove the following generalization of a result of Costa, Di Rocco and Mir\'o-Roig \cite{CM}:
\begin{thm2}(Theorem 5.4)
Let $X$, $G$ and $\mathcal{E}$ be as in Theorem 5.1. If $[X/G]$ has a tilting bundle, then so does $[\mathbb{P}(\mathcal{E})/G]$. 
\end{thm2}
In the general case where $G$ is an arbitrary algebraic group one cannot expect to have a result such as Theorem 4.2 but nevertheless, under some mild conditions, there are always semiorthogonal decompositions \cite{EL1}. As an application of the above results we provide some further evidence for a conjecture firstly formulated by Orlov \cite{O1} for schemes and extended by Ballard and Favero \cite{BF} to certain Deligne--Mumford stacks $\mathcal{X}$. It is the following dimension conjecture about the Rouquier dimension \cite{RO} of the triangulated category $D^b(\mathcal{X})$:
\begin{con2}(\cite{BF})
\textnormal{Let $\mathcal{X}$ be a smooth and tame Deligne--Mumford stack of finite type over $k$ with quasiprojective coarse moduli space, then $\mathrm{dim}(D^b(\mathcal{X}))=\mathrm{dim}(\mathcal{X})$}.
\end{con2}
Assuming $\mathrm{char}(k) \nmid \mathrm{ord}(G)$, we prove:  

\begin{thm2}(Theorem 6.8)
The dimension conjecture holds for:
\begin{itemize}
 \item[\bf(i)] quotient stacks $[\mathbb{P}(\mathcal{E})/G]$ as in Theorem 5.4, provided $[X/G]$ has a tilting bundle and $k$ is perfect.
	\item[\bf(ii)] quotient stacks $[X/G]$ over a perfect field $k$, where $X$ is a Brauer--Severi variety corresponding to a central simple algebra $A$ and $G\subset \mathrm{Aut}(X)=A^*/k^{\times}$ a finite subgroup such that the action lifts to an action of $A^*$.  
	\item[\bf(iii)] quotient stacks $[\mathrm{Grass}(d,n)/G]$ over an algebraically closed field $k$ of characteristic zero, provided $G\subset\mathrm{PGL}_n(k)$ is a finite subgroup acting linearly on $\mathrm{Grass}(d,n)$.
\item[\bf (iv)] $G$-Hilbert schemes $\mathrm{Hilb}_G(\mathbb{P}^n)$ over an algebraically closed field $k$ of characteristic zero, provided $n\leq 3$ and $G\subset \mathrm{PGL}_{n+1}(k)$ is a finite subgroup acting linearly on $\mathbb{P}^n$ and $\omega_{\mathbb{P}^n}$ is locally trivial in $\mathrm{Coh}_G(\mathbb{P}^n)$.
	\end{itemize} 
\end{thm2}
{\small \textbf{Acknowledgement}. This paper is based on a part of my Ph.D. thesis which was supervised by Stefan Schr\"oer whom I would like to thank for a lot of comments and helpful discussions. I also like to thank Markus Perling, Nicolas Perrin and Alexander Samokhin for fruitful conversations. Finally, I would like to thank the referee for the careful review and the valuable comments, which provided insights that helped improve the paper.\\ 

{\small \textbf{Conventions}. Throughout this work $k$ is an arbitrary field unless stated otherwise. Moreover, for a finite group $G$ acting on a $k$-scheme $X$, we always assume $\mathrm{char}(k) \nmid \mathrm{ord}(G)$.

\section{Generalities on equivariant derived categories}

Let $X$ be a quasiprojective $k$-scheme and $G$ a finite group acting on $X$. A sheaf $\mathcal{F}$ on $X$ is called \emph{invariant} if there are isomorphisms $g^*\mathcal{F}\simeq \mathcal{F}$ for all $g\in G$. The full additive subcategory of invariant coherent sheaves however is not abelian and thus not suitable for forming derived categories. One has to pass to $G$-linearizations. 

A \emph{G-linearization}, also called an \emph{equivariant structure}, on $\mathcal{F}$ is given by isomorphisms $\lambda_g\colon \mathcal{F}\stackrel{\sim}\rightarrow g^*\mathcal{F}$ for all $g\in G$ subject to $\lambda_1=\mathrm{id}_{\mathcal{F}}$ and $\lambda_{gh}=h^*\lambda_g\circ \lambda_h$. In the present work we also call such sheaves \emph{equivariant sheaves}. Equivariant sheaves are therefore pairs $(\mathcal{F}, \lambda)$, consisting of a sheaf $\mathcal{F}$ on $X$ and a choice of an equivariant structure $\lambda$. Clearly, an equivariant sheaf is invariant, but the other implication is wrong in general. There is an obstruction for an invariant sheaf against having an equivariant structure in terms of the second group cohomology $H^2(G,k^*)$ (see \cite{PL}, Lemma 1).

\begin{rema}
\textnormal{For a definition of linearization in the case where an arbitrary algebraic group acts on an arbitrary scheme we refer to \cite{BL}, \cite{EL1} or \cite{EL2}.}
\end{rema}  
If $(\mathcal{F},\lambda)$ and $(\mathcal{G},\mu)$ are two equivariant sheaves on $X$, the vector space $\mathrm{Hom}(\mathcal{F},\mathcal{G})$ becomes a $G$-representation via $g\cdot f:=(\mu_g)^{-1}\circ g^*f\circ \lambda_g$ for $f\colon \mathcal{F}\rightarrow \mathcal{G}$. The equivariant quasi-coherent respectively coherent sheaves together with $G$-invariant morphisms $\mathrm{Hom}_G(\mathcal{F},\mathcal{G}):=\mathrm{Hom}(\mathcal{F},\mathcal{G})^G$ form abelian categories with enough injectives (see \cite{BR}, \cite{PL}, \cite{PL1}) which we denote by $\mathrm{Qcoh}_G(X)$ respectively $\mathrm{Coh}_G(X)$. We put $D_G(\mathrm{Qcoh}(X)):=D(\mathrm{Qcoh}_G(X))$ and $D^b_G(X):=D^b(\mathrm{Coh}_G(X))$. 

Let $X$ and $Y$ be quasiprojective $k$-schemes on which the finite group $G$ acts. A \emph{$G$-morphisms} between $X$ and $Y$ is given by a morphism $\phi\colon X\rightarrow Y$ such that $\phi\circ g=g\circ \phi$ for all $g\in G$. Then we have the pullback $\phi^*\colon \mathrm{Coh}_G(Y)\rightarrow \mathrm{Coh}_G(X)$ and the pushforward $\phi_*\colon \mathrm{Coh}_G(X)\rightarrow \mathrm{Coh}_G(Y)$. The functors $\phi^*$ and $\phi_*$ are adjoint; analogously for $\mathbb{L}\phi^*$ and $\mathbb{R}\phi_*$. For $(\mathcal{F},\lambda), (\mathcal{G},\mu)\in \mathrm{Coh}_G(X)$ there is a canonical equivariant structure on $\mathcal{F}\otimes\mathcal{G}$ coming from the maps $\lambda_g\otimes\mu_g$ (see \cite{BL}).

By definition, objects of $D^b_G(X)$ are bounded complexes of equivariant coherent sheaves. It is clear that each such complex defines an equivariant structure on the corresponding object of $D^b(X)$. Now let $\mathcal{C}$ be the category of equivariant objects of $D^b(X)$, i.e. complexes $\mathcal{F}^{\bullet}$ with isomorphisms $\lambda_g\colon \mathcal{F}^{\bullet}\stackrel{\sim}\rightarrow g^*\mathcal{F}^{\bullet}$ satisfying $\lambda_{gh}=h^*\lambda_g\circ \lambda_h$. This category is in fact triangulated and it is a natural fact that $D^b_G(X)$ and $\mathcal{C}$ are equivalent (see \cite{C}, Proposition 4.5 or \cite{EL2}). 

There is also another description of the derived categories needed in the present work. Consider the global quotient stack $[X/G]$, produced by an action of a finite group $G$ on $X$ (see \cite{V}, Example 7.17). The quasi-coherent sheaves on $[X/G]$ are equivalent to equivariant quasi-coherent sheaves on $X$ (see \cite{V}, Example 7.21). Henceforth, the abelian categories $\mathrm{Qcoh}([X/G])$ and $\mathrm{Qcoh}_G(X)$ are equivalent and therefore give rise to equivalent derived categories $D_G(\mathrm{Qcoh}(X))\simeq D(\mathrm{Qcoh}([X/G]))$. For any two objects $\mathcal{F}^{\bullet}, \mathcal{G}^{\bullet}\in D_G(\mathrm{Qcoh}(X))$ we write $\mathrm{Hom}_G(\mathcal{F}^{\bullet},\mathcal{G}^{\bullet}):=\mathrm{Hom}_{D_G(\mathrm{Qcoh}(X))}(\mathcal{F}^{\bullet},\mathcal{G}^{\bullet})$.

Analogously, we get $D^b_G(X)\simeq D^b(\mathrm{Coh}([X/G]))$. Note that for $X=pt$, $\mathrm{Coh}([pt/G])\simeq \mathrm{Coh}_G(pt)\simeq \mathrm{Rep}_k(G)$ is the category of finite-dimensional representations. Moreover, for a finite group $G$ with $\mathrm{char}(k) \nmid \mathrm{ord}(G)$, the functor $(-)^G:\mathrm{Coh}([pt/G])\rightarrow \mathrm{Coh}(pt), V\mapsto V^G$, is exact (see \cite{AOV}, Proposition 2.5). For arbitrary $\mathcal{F}^{\bullet}, \mathcal{G}^{\bullet}\in D^b_G(X)$, the finite group $G$ also acts on the vector space $\mathrm{Hom}(\mathcal{F}^{\bullet},\mathcal{G}^{\bullet}):=\mathrm{Hom}_{D^b(X)}(\mathcal{F}^{\bullet},\mathcal{G}^{\bullet})$. The exactness of $(-)^G$ yields
\begin{center}
$\mathrm{Hom}_G(\mathcal{F}^{\bullet},\mathcal{G}^{\bullet})\simeq \mathrm{Hom}(\mathcal{F}^{\bullet},\mathcal{G}^{\bullet})^G$. 
\end{center}
The exactness of $(-)^G$ also implies the following fact (see \cite{BFK}, Lemma 2.2.8):
\begin{lem}
Let $X$ be smooth quasiprojective $k$-scheme and $G$ a finite group acting on $X$. For arbitrary $\mathcal{F}^{\bullet},\mathcal{G}^{\bullet}\in D^b_G(X)$ the following holds for all $i\in \mathbb{Z}$:
\begin{center}
$\mathrm{Hom}_{G}(\mathcal{F}^{\bullet},\mathcal{G}^{\bullet}[i])\simeq \mathrm{Hom}(\mathcal{F}^{\bullet},\mathcal{G}^{\bullet}[i])^G$.
\end{center}
\end{lem} 

\section{Geometric tilting theory}
In this section we recall some facts of geometric tilting theory. We first recall the notions of generating and thick subcategories (see \cite{BV}, \cite{RO}).\\

Let $\mathcal{D}$ be a triangulated category and $\mathcal{C}$ a triangulated subcategory. The subcategory $\mathcal{C}$ is called \emph{thick} if it is closed under isomorphisms and direct summands. For a subset $A$ of objects of $\mathcal{D}$ we denote by $\langle A\rangle$ the smallest full thick subcategory of $\mathcal{D}$ containing the elements of $A$. 
Furthermore, we define $A^{\perp}$ to be the subcategory of $\mathcal{D}$ consisting of all objects $M$ such that $\mathrm{Hom}_{\mathcal{D}}(E[i],M)=0$ for all $i\in \mathbb{Z}$ and all elements $E$ of $A$. We say that $A$ \emph{generates} $\mathcal{D}$ if $A^{\perp}=0$. Now assume $\mathcal{D}$ admits arbitrary direct sums. An object $B$ is called \emph{compact} if $\mathrm{Hom}_{\mathcal{D}}(B,-)$ commutes with direct sums. Denoting by $\mathcal{D}^c$ the subcategory of compact objects we say that $\mathcal{D}$ is \emph{compactly generated} if the objects of $\mathcal{D}^c$ generate $\mathcal{D}$. One has the following important theorem (see \cite{BV}, Theorem 2.1.2).
\begin{thm}
Let $\mathcal{D}$ be a compactly generated triangulated category. Then a set of objects $A\subset \mathcal{D}^c$ generates $\mathcal{D}$ if and only if $\langle A\rangle=\mathcal{D}^c$.  
\end{thm}
We now give the definition of tilting objects (see \cite{BH} for a definition of tilting objects in arbitrary triangulated categories).
\begin{defi}
\textnormal{Let $k$ be a field, $X$ a quasiprojective $k$-scheme and $G$ a finite group acting on $X$. An object $\mathcal{T}^{\bullet}\in D_G(\mathrm{Qcoh}(X))$ is called \emph{tilting object} on $[X/G]$ 
if the following hold:
\begin{itemize}
      \item[\bf (i)] Ext vanishing: $\mathrm{Hom}_G(\mathcal{T}^{\bullet},\mathcal{T}^{\bullet}[i])=0$ for $i\neq0$.
     \item[\bf (ii)] Generation: If $\mathcal{N}^{\bullet}\in D_G(\mathrm{Qcoh}(X))$ satisfies $\mathbb{R}\mathrm{Hom}_G(\mathcal{T}^{\bullet},\mathcal{N}^{\bullet})=0$, then $\mathcal{N}^{\bullet}=0$.
		\item[\bf (iii)] Compactness: $\mathrm{Hom}_G(\mathcal{T}^{\bullet},-)$ commutes with direct sums.
		\end{itemize}}
\end{defi}
Below we state the well-known equivariant tilting correspondence (see \cite{BR}, Theorem 3.1.1). It is a direct application of a more general result on triangulated categories (see \cite{KE}, Theorem 8.5). We denote by $\mathrm{Mod}(A)$ the category of right $A$-modules and by $D^b(A)$ the bounded derived category of finitely generated right $A$-modules. Furthermore, $\mathrm{perf}(A)\subset D(\mathrm{Mod}(A)) $ denotes the full triangulated subcategory of perfect complexes, those quasi-isomorphic to a bounded complexes of finitely generated projective right $A$-modules.
\begin{thm}
Let $X$ be a quasiprojective $k$-scheme and $G$ a finite group acting on $X$. Suppose we are given a tilting object $\mathcal{T}^{\bullet}$ on $[X/G]$ 
and let $A=\mathrm{End}_G(\mathcal{T}^{\bullet})$. Then the following hold:
\begin{itemize}
      \item[\bf (i)] The functor $\mathbb{R}\mathrm{Hom}_G(\mathcal{T}^{\bullet},-)\colon D_G(\mathrm{Qcoh}(X))\rightarrow D(\mathrm{Mod}(A))$ is an equivalence. 
      \item[\bf (ii)] If $X$ is smooth and $\mathcal{T}^{\bullet}\in D^b_G(X)$, this equivalence restricts to an equivalence $D^b_G(X)\stackrel{\sim}\rightarrow \mathrm{perf}(A)$.
			
			\item[\bf (iii)] If the global dimension of $A$ is finite, then $\mathrm{perf}(A)\simeq D^b(A)$. 
\end{itemize} 
\end{thm}
 \begin{rema}
\textnormal{If $X$ is a smooth projective $k$-scheme and $G=1$, the derived category $D(\mathrm{Qcoh}(X))$ is compactly generated and the compact objects are exactly $D^b(X)$ (see \cite{BV}). In this case, a compact object $\mathcal{T}^{\bullet}$ generates $D(\mathrm{Qcoh}(X))$ if and only if $\langle\mathcal{T}^{\bullet}\rangle=D^b(X)$. Since the natural functor $D^b(X)\rightarrow D(\mathrm{Qcoh}(X))$ is fully faithful (see \cite{HUY}), a compact object $\mathcal{T}^{\bullet}\in D(\mathrm{Qcoh}(X))$ is a tilting object if and only if $\langle\mathcal{T}^{\bullet}\rangle=D^b(X)$ and $\mathrm{Hom}_{D^b(X)}(\mathcal{T}^{\bullet},\mathcal{T}^{\bullet}[i])=0$ for $i\neq 0$. If the tilting object $\mathcal{T}^{\bullet}$ is a coherent sheaf and $\mathrm{gldim}(\mathrm{End}(\mathcal{T}^{\bullet}))<\infty$, the above definition coincides with the definition of a tilting sheaf given in \cite{B}. In this case the tilting object is called \emph{tilting sheaf} on $X$. If it is a locally free sheaf we simply say that $\mathcal{T}$ is a \emph{tilting bundle}. Theorem 3.3 then gives the classical tilting correspondence as first proved by Bondal \cite{BO} and later extended by Baer \cite{B}}.
\end{rema}
The next observation shows that in Theorem 3.3 the smoothness of $X$ already implies the finiteness of the global dimension of $A$.
\begin{prop}
Let $X$, $G$ and $\mathcal{T}^{\bullet}$ be as in Theorem 3.3. If $X$ is smooth and projective, then $A=\mathrm{End}_G(\mathcal{T}^{\bullet})$ has finite global dimension and therefore the equivalence (i) of Theorem 3.3 restricts to an equivalence $D^b_G(X)\stackrel{\sim}\rightarrow D^b(A)$. 
\end{prop}
\begin{proof}
Imitating the proof of Theorem 7.6 in \cite{HV}, we argue as follows: For two finitely generated right $A$-modules $M$ and $N$, the equivalence $\psi:=\mathbb{R}\mathrm{Hom}_G(\mathcal{T}^{\bullet},-)\colon D^b_G(X)\rightarrow \mathrm{perf}(A)$ (see Theorem 3.3 (ii)) yields
\begin{center}
$\mathrm{Ext}^i_A(M,N)\simeq \mathrm{Hom}_G(\psi^{-1}(M),\psi^{-1}(N)[i])\simeq \mathrm{Hom}(\psi^{-1}(M),\psi^{-1}(N)[i])^G=0$ 
\end{center}
for $i\gg 0$, since $X$ is smooth. Indeed, this follows from the local-to-global spectral sequence, Grothendieck vanishing Theorem and Lemma 2.2. As $X$ is projective, $A=\mathrm{End}_G(\mathcal{T}^{\bullet})$ is a finite-dimensional $k$-algebra and hence a noetherian ring. But for noetherian rings the vanishing of $\mathrm{Ext}^i_A(M,N)$ for $i\gg 0$ for any two finitely generated $A$-modules $M$ and $N$ suffices to conclude that the global dimension of $A$ has to be finite.   
\end{proof}
In the literature, instead of the tilting object $\mathcal{T}^{\bullet}$ one often studies the set $\mathcal{E}^{\bullet}_1,...,\mathcal{E}^{\bullet}_n$ of its indecomposable pairwise non-isomorphic direct summands. There is a special case where all the summands form a so-called full strongly exceptional collection. Closely related to the notion of a full strongly exceptional collection is that of a semiorthogonal decomposition. We recall the definitions and follow here \cite{O2}.
\begin{defi}
\textnormal{Let $X$ and $G$ be as in Definition 3.2. An object $\mathcal{E}^{\bullet}\in D^b_G(X)$ is called \emph{exceptional} if $\mathrm{Hom}_G(\mathcal{E}^{\bullet},\mathcal{E}^{\bullet}[l])=0$ when $l\neq 0$, and $\mathrm{Hom}_G(\mathcal{E}^{\bullet},\mathcal{E}^{\bullet})=k$. An \emph{exceptional collection} in $D^b_G(X)$ is a sequence of exceptional objects $\mathcal{E}^{\bullet}_1,...,\mathcal{E}^{\bullet}_n$ satisfying $\mathrm{Hom}_G(\mathcal{E}^{\bullet}_i,\mathcal{E}^{\bullet}_j[l])=0$ for all $l\in \mathbb{Z}$ if $i>j$.}

\textnormal{The exceptional collection is called \emph{strongly exceptional} if in addition $\mathrm{Hom}_G(\mathcal{E}^{\bullet}_i,\mathcal{E}^{\bullet}_j[l])=0$ for all $i$ and $j$ when $l\neq 0$. Finally, we say the exceptional collection is \emph{full} if the smallest full thick subcategory containing all $\mathcal{E}^{\bullet}_i$ equals $D^b_G(X)$.}
\end{defi}
A generalization is the notion of a semiorthogonal decomposition of $D_G^b(X)$.
Recall, a full triangulated subcategory $\mathcal{D}$ of $D^b_G(X)$ is called \emph{admissible} if the inclusion $\mathcal{D}\hookrightarrow D^b_G(X)$ has a left and right adjoint functor.
\begin{defi}
\textnormal{Let $X$ and $G$ be as in Definition 3.2. A sequence $\mathcal{D}_1,...,\mathcal{D}_n$ of full triangulated subcategories of $D^b_G(X)$ is called \emph{semiorthogonal} if all $\mathcal{D}_i\subset D^b_G(X)$ are admissible and $\mathcal{D}_j\subset \mathcal{D}_i^{\perp}=\{\mathcal{F}^{\bullet}\in D_G^b(X)\mid \mathrm{Hom}_G(\mathcal{G}^{\bullet},\mathcal{F}^{\bullet})=0$, $\forall$ $ \mathcal{G}^{\bullet}\in\mathcal{D}_i\}$ for $i>j$.}

\textnormal{Such a sequence defines a \emph{semiorthogonal decomposition} of $D^b_G(X)$ if the smallest full thick subcategory containing all $\mathcal{D}_i$ equals $D^b_G(X)$.}
\end{defi}
For a semiorthogonal decomposition of $D_G^b(X)$, we write $D_G^b(X)=\langle\mathcal{D}_1,...,\mathcal{D}_r\rangle$.
\begin{exam}
\textnormal{It is an easy exercise to show that a full exceptional collection $\mathcal{E}^{\bullet}_1,...,\mathcal{E}^{\bullet}_n$ in $D_G^b(X)$ gives rise to a semiorthogonal decomposition $D_G^b(X)=\langle\mathcal{D}_1,...,\mathcal{D}_n\rangle$, where $\mathcal{D}_i=\langle \mathcal{E}^{\bullet}_i\rangle$ (see \cite{HUY}, Example 1.60).} 
\end{exam} 
The direct sum of the exceptional objects in a full strongly exceptional collection is a tilting object but the pairwise non-isomorphic indecomposable direct summands of a tilting object in general cannot be arranged into a full strongly exceptional collection. However, if the pairwise non-isomorphic indecomposable direct summands are invertible sheaves, they give rise to a full strongly exceptional collection. Exceptional collections and semiorthogonal decompositions were intensively studied and we know quite a lot of examples of schemes admitting full exceptional collections or semiorthogonal decompositions. For an overview we refer to \cite{BOO} and \cite{KU}.\\
\section{Tilting objects on $[X/G]$}
Let $G$ be a finite group acting on a smooth projective $k$-scheme $X$. One has the $G$-morphism $f:X\rightarrow \mathrm{Spec}(k)$, with $G$ acting trivially on $\mathrm{Spec}(k)$. For a representation $W$ of $G$, the sheaf $f^{*}W=\mathcal{O}_X\otimes W$ has a natural equivariant structure. We shortly write $W$ for $f^{*}W=\mathcal{O}_X\otimes W$ as an object of $\mathrm{Coh}_G(X)$. With this notation we prove the following theorem.
\begin{thm}
Let $X$ be a smooth projective $k$-scheme and $G$ a finite group acting on $X$. Suppose there is a $\mathcal{T}^{\bullet}\in D_G(\mathrm{Qcoh}(X))$ which, considered as an object in $D(\mathrm{Qcoh}(X))$, is a tilting object on $X$. Denote by $W_j$ the irreducible representations of $G$, then $\bigoplus_j\mathcal{T}^{\bullet}\otimes W_j$ is a tilting object on $[X/G]$.
\end{thm}
\begin{proof}
As $\mathcal{T}^{\bullet}$ is a tilting object on $X$, it is compact by definition and hence $\mathcal{T}^{\bullet}\in D^b(X)$ (see \cite{BV}). Let $\mu_j$ be the equivariant structure on $W_j$, then $\mathcal{T}^{\bullet}\otimes W_j$ is equipped with a natural equivariant structure, say $\lambda$. Taking the direct sum gives a natural equivariant structure $\tilde{\lambda}$ on $\mathcal{T}_G:=\bigoplus_j\mathcal{T}^{\bullet}\otimes W_j$ and hence $\mathcal{T}_G\in D^b_G(X)$. So $\mathcal{T}_G$ is a compact object of $D_G(\mathrm{Qcoh}(X))$. Recall that $D^b_G(X)\simeq D^b([X/G])$. For every $i\in \mathbb{Z}$ one has canonical isomorphisms on $X$ 
\begin{eqnarray}
\mathrm{Hom}(\mathcal{T}^{\bullet}\otimes W_l,\mathcal{T}^{\bullet}\otimes W_m[i])\simeq \mathrm{Hom}(\mathcal{T}^{\bullet},\mathcal{T}^{\bullet}[i])\otimes \mathrm{Hom}(W_l,W_m).
\end{eqnarray}  
As $G$ is finite, Lemma 2.2 applies and we get with (1): 
\begin{center}
$\mathrm{Hom}_G(\mathcal{T}^{\bullet}\otimes W_l,\mathcal{T}^{\bullet}\otimes W_m[i])\simeq (\mathrm{Hom}(\mathcal{T}^{\bullet},\mathcal{T}^{\bullet}[i])\otimes \mathrm{Hom}(W_l,W_m))^G$. 
\end{center} Since $\mathcal{T}^{\bullet}$ is a tilting object for $X$, we have $\mathrm{Hom}(\mathcal{T}^{\bullet},\mathcal{T}^{\bullet}[i])=0$ for $i\neq 0$ and therefore $\mathrm{Hom}_G(\mathcal{T}^{\bullet}\otimes W_l,\mathcal{T}^{\bullet}\otimes W_m[i])=0$ for $i\neq 0$.
This implies $\mathrm{Hom}_G(\mathcal{T}_G,\mathcal{T}_G[i])=0$ for $i\neq 0$ and hence the Ext vanishing holds true. 

To see that $\mathcal{T}_G$ generates $D_G(\mathrm{Qcoh}(X))$, we take an object  
$\mathcal{F}^{\bullet}\in D_G(\mathrm{Qcoh}(X))$ and assume $\mathbb{R}\mathrm{Hom}_G(\mathcal{T}_G,\mathcal{F}^{\bullet})=0$, i.e. $\mathrm{Hom}_G(\mathcal{T}_G,\mathcal{F}^{\bullet}[i])=0$ for all $i\in\mathbb{Z}$.

Since $\mathcal{T}_G=\bigoplus_j\mathcal{T}^{\bullet}\otimes W_j$, we have $\mathrm{Hom}_G(\mathcal{T}^{\bullet}\otimes W_j,\mathcal{F}^{\bullet}[i])=0$ for all $i\in\mathbb{Z}$ and all irreducible representations $W_j$. From 
\begin{center}
$\mathrm{Hom}_G(\mathcal{T}^{\bullet}\otimes W_j,\mathcal{F}^{\bullet}[i])\simeq \mathrm{Hom}_G(W_j,\mathbb{R}\mathrm{Hom}(\mathcal{T}^{\bullet},\mathcal{F}^{\bullet}[i]))=0$
\end{center} 
we conclude that $\mathbb{R}\mathrm{Hom}(\mathcal{T}^{\bullet},\mathcal{F}^{\bullet}[i])$ contains no copy of any irreducible representation $W_m$ and so must be zero. Since $\mathcal{T}^{\bullet}$ is a tilting object on $X$ and therefore generates $D(\mathrm{Qcoh}(X))$, we find $\mathcal{F}^{\bullet}=0$. This shows that $\mathcal{T}_G$ generates $D_G(\mathrm{Qcoh}(X))$. 
\end{proof}
If $\mathrm{char}(k) \nmid \mathrm{ord}(G)$, the regular representation $k[G]$ is the direct sum of multiple copies of the irreducible representations of $G$. More precise, as a $G$-representation $k[G]=\bigoplus_i W^{\oplus \mathrm{dim}(W_i)}_i$, where $W_i$ are the irreducible representations. This follows from the Artin--Wedderburn Theorem as $k[G]$ is semi-simple.
\begin{thm}
Let $X$ be a smooth projective $k$-scheme and $G$ a finite group acting on $X$. Suppose there is a $\mathcal{T}^{\bullet}\in D_G(\mathrm{Qcoh}(X))$ which, considered as an object in $D(\mathrm{Qcoh}(X))$, is a tilting object on $X$. Then $\mathcal{T}^{\bullet}\otimes k[G]$ is a tilting object on $[X/G]$.
\end{thm}
\begin{proof}
Repeating the proof of Theorem 4.1 one verifies that $\mathcal{T}^{\bullet}\otimes k[G]=\bigoplus_i \mathcal{T}^{\bullet}\otimes W^{\oplus \mathrm{dim}(W_i)}_i\in D^b_G(X)$ generates $D_G(\mathrm{Qcoh}(X))$ and has no higher self extensions.
\end{proof}
Note that different equivariant structures on the tilting object $\mathcal{T}^{\bullet}\in D^b(X)$ give rise to different tilting objects on $[X/G]$.
\begin{exam}
\textnormal{We know that $\mathcal{T}=\bigoplus^n_{i=0}\mathcal{O}_{\mathbb{P}^n}(i)$ is a tilting bundle on $\mathbb{P}^n$ [5]. Consider a finite subgroup $G\subset\mathrm{Aut}_k(\mathbb{P}^n)\simeq \mathrm{PGL}_{n+1}(k)$ and the stack $[\mathbb{P}^n/G]$. Assume that $G$ acts linearly on $\mathbb{P}^n$, i.e the action lifts to an action of $GL_{n+1}(k)$. Then any invertible sheaf on $\mathbb{P}^n$ admits an equivariant structure. So by choosing such on each $\mathcal{O}_{\mathbb{P}^n}(i)$, $\mathcal{T}\in D^b_G(\mathbb{P}^n)$. Now Theorem 4.1 gives a tilting bundle on $[\mathbb{P}^n/G]$ (see also \cite{BR}, Theorem 3.2.1)}. 
\end{exam}
\begin{exam}
\textnormal{Let $\mathrm{Grass}_k(d,n)$ be the Grassmannian over an algebraically closed field $k$ of characteristic zero. For $2d\neq n$ one has $\mathrm{Aut}_k(\mathrm{Grass}_k(d,n))=\mathrm{PGL}_n(k)$. Let $G\subset \mathrm{PGL}_n(k)$ be a finite group acting linearly on $\mathrm{Grass}_k(d,n)$, i.e. the action lifts to an action of $\mathrm{GL}_n(k)$. Then the tautological sheaf $\mathcal{S}$ of $\mathrm{Grass}_k(d,n)$ and, due to functoriality, the Schur modules $\Sigma^{\lambda}(\mathcal{S})$ admit a natural equivariant structure. So $\bigoplus_{\lambda}\Sigma^{\lambda}(\mathcal{S})\in D^b_G(\mathrm{Grass}_k(d,n))$. As $\bigoplus_{\lambda}\Sigma^{\lambda}(\mathcal{S})$ is a tilting bundle on $\mathrm{Grass}_k(d,n)$ (see \cite{KA}), Theorem 4.1 gives a tilting bundle on $[\mathrm{Grass}_k(d,n)/G]$}.
\end{exam}
Let $X$ be a smooth projective $k$-scheme and $G$ a finite group acting on $X$. For a field extension $k\subset E$ we set $X_E:=X\otimes_k E$ and $G_E:=G\otimes_k E$. Since $G$ acts on $X$, the group $G_E$ acts on $X_E$. Suppose there is an object $\mathcal{T}^{\bullet}\in D^b_G(X)$ such that $\mathcal{T}^{\bullet}\otimes_k E$ is a tilting object on $X_E$. Below we prove that in fact $\mathcal{T}^{\bullet}$ is a tilting object on $[X/G]$. We first need the following lemma, essentially proved in \cite{BLU}. For convenience of the reader we give a proof. 
\begin{lem}
Let $X$ be a smooth projective $k$-scheme and $k\subset E$ a field extension. For a given object $\mathcal{T}^{\bullet}\in D^b(X)$, suppose that $\mathcal{T}^{\bullet}\otimes_k E$ is a tilting object on $X\otimes_k E$. Then $\mathcal{T}^{\bullet}$ is a tilting object on $X$.
\end{lem}
\begin{proof}
Let $v:X\otimes_k E\rightarrow X$ be the projection. By assumption $v^*\mathcal{T}^{\bullet}=\mathcal{T}^{\bullet}\otimes_k E$ is a tilting object on $X\otimes_k E$. 
We calculate $\mathrm{Hom}(\mathcal{T}^{\bullet},\mathcal{T}^{\bullet}[i])$. For this, we consider the following base change diagram
\begin{displaymath}
\begin{xy}
  \xymatrix{
      X\otimes_k E\ar[r]^{v} \ar[d]_{q}    &   X \ar[d]^{p}                   \\
      \mathrm{Spec}(E) \ar[r]^{u}             &   \mathrm{Spec}(k)             
  }
\end{xy}
\end{displaymath}
Let $\mathcal{E}^{\bullet}$ be a bounded complex of locally free sheaves and $\mathcal{F}^{\bullet}\in D(\mathrm{Qcoh}(X))$ arbitrary. Then flat base change (see \cite{HUY}, p.85 (3.18)) yields isomorphisms 
\begin{eqnarray*}
u^*(\mathbb{R}\mathrm{Hom}(\mathcal{E}^{\bullet},\mathcal{F}^{\bullet}))&\simeq &u^*\mathbb{R}p_*\mathbb{R}\mathcal{H}om(\mathcal{E}^{\bullet},\mathcal{F}^{\bullet})\\
 &\simeq &\mathbb{R}q_*v^*\mathbb{R}\mathcal{H}om(\mathcal{E}^{\bullet},\mathcal{F}^{\bullet})\\
&\simeq &\mathbb{R}q_*v^*({\mathcal{E}^{\bullet}}^{\vee}\otimes\mathcal{F}^{\bullet})\\
&\simeq &\mathbb{R}q^*\mathbb{R}\mathcal{H}om(v^*\mathcal{E}^{\bullet},v^*\mathcal{F}^{\bullet})\\
&\simeq &\mathbb{R}\mathrm{Hom}(v^*\mathcal{E}^{\bullet},v^*\mathcal{F}^{\bullet}).
\end{eqnarray*}
This implies 
\begin{center}
$\mathrm{Hom}(v^*\mathcal{T}^{\bullet},v^*\mathcal{T}^{\bullet}[i])\simeq \mathrm{Hom}(\mathcal{T}^{\bullet},\mathcal{T}^{\bullet}[i])\otimes_k E=0$
for $i\neq 0$, 
\end{center}
since $v^*\mathcal{T}^{\bullet}$ is a tilting object on $X\otimes_k E$. Hence $\mathrm{Hom}(\mathcal{T}^{\bullet},\mathcal{T}^{\bullet}[i])=0$ for $i\neq 0$ and therefore Ext vanishing holds. 

For the generation property of $\mathcal{T}^{\bullet}$, we take an object $\mathcal{F}^{\bullet}\in D(\mathrm{Qcoh}(X))$ and assume $\mathbb{R}\mathrm{Hom}(\mathcal{T}^{\bullet},\mathcal{F}^{\bullet})=0$. The above isomorphisms obtained from flat base change yield
\begin{eqnarray*}
0=u^*(\mathbb{R}\mathrm{Hom}(\mathcal{T}^{\bullet},\mathcal{F}^{\bullet}))\simeq \mathbb{R}\mathrm{Hom}(v^*\mathcal{T}^{\bullet},v^*\mathcal{F}^{\bullet}).
\end{eqnarray*}
Since $v^*\mathcal{T}^{\bullet}$ is a tilting object on $X\otimes_k E$, we have $v^*\mathcal{F}^{\bullet}=0$. As $v$ is a faithfully flat morphism, $\mathcal{F}^{\bullet}=0$ and hence $\mathcal{T}^{\bullet}$ generates $D(\mathrm{Qcoh}(X))$. Finally, since $X$ is smooth, the global dimension of $\mathrm{End}(\mathcal{T}^{\bullet})$ is finite. This completes the proof.
\end{proof}
\begin{prop}
Let $X$ be a smooth projective $k$-scheme, $\mathcal{T}^{\bullet} \in D^b_G(X)$ and $k\subset E$ a field extension. Considering $\mathcal{T}^{\bullet}$ as an object in $D^b(X)$, suppose $\mathcal{T}^{\bullet}\otimes_k E$ is a tilting object on $X_E$. Then $\mathcal{T}^{\bullet}\otimes k[G]$ is a tilting object on $[X/G]$. 
\end{prop}
\begin{proof}
Since $\mathcal{T}^{\bullet}\otimes_k E$ is a tilting object on $X_E$, Lemma 4.5 implies that $\mathcal{T}^{\bullet}$ is a tilting object on $X$. As $\mathcal{T}^{\bullet}\in D^b_G(X)$, Theorem 4.2 yields the assertion. 
\end{proof}
\begin{exam}
\textnormal{Let $X$ be a $n-1$-dimensional Brauer--Severi variety (see \cite{NO} and references therein for details). Such a Brauer--Severi variety is associated to a central simple $k$-algebra $A$ of dimension $n^2$ in the following way: Consider the set of all left ideals $I$ of $A$ of rank $n$. This set can be given the structure of a smooth projective $k$-scheme by embedding it as a closed subscheme of $\mathrm{Grass}(n,n^2)$ defined by the relations stating that each $I$ is a left ideal. Indeed, there is a natural one-to-one correspondence between central simple algebras of dimension $n^2$ and Brauer--Severi varieties of dimension $n-1$ via Galois cohomology (see \cite{A}). As $X$ being a closed subscheme of the Grassmannian, it is endowed with a tautological sheaf $\mathcal{V}$ of rank $n$. This sheaf has a natural $A$ action. Note  that $A\otimes_k \bar{k}\simeq M_n(\bar{k})$. As $\mathrm{Aut}(X)=\mathrm{Aut}(A)= A^*/{k^{\times}}$ by the Skolem--Noether Theorem, we see that if the action of a finite subgroup $G\subset \mathrm{Aut}(X)$ lifts to an action of $A^*$, the tautological sheaf $\mathcal{V}$ has a natural equivariant structure. So $\bigoplus^n_{i=0}\mathcal{V}^{\otimes i}\in D^b_G(X)$. Since $\mathcal{V}^{\otimes i}\otimes_k \bar{k}\simeq \mathcal{O}_{\mathbb{P}^n}(-i)^{\oplus (n+1)^i}$ \cite{BLU}, the sheaf $(\bigoplus^n_{i=0}\mathcal{V}^{\otimes i})\otimes_k \bar{k}$ is a tilting bundle on $\mathbb{P}^n$. Proposition 4.6 gives a tilting bundle on $[X/G]$}.
\end{exam}
The next proposition shows that Example 4.7 cannot be obtained from the results given in \cite{EL1}, where tilting bundles on $[X/G]$ are constructed from full strongly exceptional collections. In particular it gives an example of a $k$-scheme admitting a tilting object, but not a full strongly exceptional collection.
\begin{prop}
Let $X\neq \mathbb{P}^1$ be a $1$-dimensional Brauer--Severi variety. Then $D^b(X)$ does not admit a full strongly exceptional collection.
\end{prop}
\begin{proof}
We first prove that $D^b(X)$ does not admit a full strongly exceptional collection consisting of coherent sheaves. Denote by $\mathcal{V}$ the tautological sheaf on $X$. Since $\mathcal{T}=\mathcal{O}_X\oplus \mathcal{V}$ is a tilting bundle on $X$ (see \cite{BLU}), the Grothendieck group ${K}_0(X)$ is a free abelian group of rank two. So assume $\mathcal{E}_1$ and $\mathcal{E}_2$ are coherent sheaves on $X$ forming a full strongly exceptional collection. In particular, $\mathrm{End}(\mathcal{E}_i)=k$ and therefore $\mathrm{End}(\mathcal{E}_i\otimes_k \bar{k})=\bar{k}$. 
We see that $\mathcal{E}_1\otimes_k \bar{k}$ and $\mathcal{E}_2\otimes_k \bar{k}$ are simple coherent sheaves on $X\otimes_k \bar{k}\simeq\mathbb{P}^1$. A simple sheaf $\mathcal{F}$ on $\mathbb{P}^1$ here means $\mathrm{End}(\mathcal{F})= \bar{k}$. Now simple coherent sheaves on the projective line are known to be invertible sheaves or skyscraper sheaves supported on a closed point. So $\mathcal{E}_i\otimes_k \bar{k}$ has to be isomorphic to either $\mathcal{O}_{\mathbb{P}^1}(n)$ or $\bar{k}(x)$. Note that every invertible sheaf on $\mathbb{P}^1$ coming from an invertible sheaf on the Brauer--Severi variety $X$ is of the form $\mathcal{O}_{\mathbb{P}^1}(2n)$ (see \cite{NO}, Section 6). There are two cases that have to be considered:
\begin{itemize}
      \item[\bf (i)] Assume both $\mathcal{E}_1\otimes_k \bar{k}$ and $\mathcal{E}_2\otimes_k \bar{k}$ are invertible sheaves. In this case $\mathcal{E}_1\otimes_k \bar{k}=\mathcal{O}_{\mathbb{P}^1}(2n)$ and $\mathcal{E}_2\otimes_k \bar{k}=\mathcal{O}_{\mathbb{P}^1}(2m)$. Without loss of generality assume $\mathcal{E}_1\otimes_k\bar{k}=\mathcal{O}_{\mathbb{P}^1}$ and $\mathcal{E}_2\otimes_k\bar{k}=\mathcal{O}_{\mathbb{P}^1}(2n)$, with $n>0$. But then we have
\begin{eqnarray*}
\mathrm{Ext}^1(\mathcal{O}_{\mathbb{P}^1}(2n),\mathcal{O}_{\mathbb{P}^1})\simeq H^1(X,\mathcal{O}_{\mathbb{P}^1}(-2n))\neq 0. 
\end{eqnarray*}

			\item[\bf (ii)] Now assume that at least one of the sheaves $\mathcal{E}_i\otimes_k \bar{k}$ is a skyscraper sheaf. Without loss of generality assume $\mathcal{E}_1\otimes_k \bar{k}= \bar{k}(x)$. Then 
			\begin{center}
	$\mathrm{Ext}^1(\bar{k}(x),\bar{k}(x))\simeq T_x$,	
		\end{center} where $T_x$ is the tangent space of $\mathbb{P}^1$ in $x$ (see \cite{HUY}, Example 11.9) that clearly is non-zero. 
\end{itemize}		
We see that $D^b(X)$ does not admit a full strongly exceptional collection consisting of coherent sheaves. To conclude that $D^b(X)$ does not admit a full strongly exceptional collection consisting of arbitrary objects, we note that $\mathrm{Coh}(X)$ is hereditary. According to \cite{KE}, Subsection 2.5, every object $\mathcal{G}\in D^b(X)$ is of the form $\bigoplus_{i\in \mathbb{Z}}H^i(\mathcal{G})[i]$. Since exceptional objects are indecomposable, the exceptional objects in $D^b(X)$ are just shifts of exceptional coherent sheaves. With the arguments from above this finally implies that $D^b(X)$ does not admits a full strongly exceptional collection.  	
\end{proof}

\begin{con2}
\textnormal{Let $X\neq \mathbb{P}^n$ be a $n$-dimensional Brauer--Severi variety. Then $D^b(X)$ does not admit a full strongly exceptional collection}.
\end{con2}
Theorem 4.1 states more or less that the stack $[X/G]$ has a tilting object if $X$ has one. Below we will see that the other implication is wrong in general (see Example 4.10). For this, we roughly recall the derived McKay correspondence and refer to the work of Bridgeland, King and Reid \cite{BK} for details.

Let $k$ be an algebraically closed field of characteristic zero and $X$ a quasiprojective $k$-scheme. Furthermore, let $G$ be a finite subgroup of $\mathrm{Aut}(X)$. Note that the quotient scheme $X/\!\!/G$ is usually singular. The main idea of McKay correspondence is to find a certain "nice" resolution of $X/\!\!/G$ and to relate the geometry of the resolution to that of $X/\!\!/G$. Recall, a resolution of singularities $f\colon \tilde{X}\rightarrow X$ of a given non-singular $X$ is called \emph{crepant} if $f^*\omega_X=\omega_{\tilde{X}}$. Whether such resolutions exist is a difficult question and closely related to the minimal model program. 

Now denote by $\mathrm{Hilb}_G(X)$ the $G$-Hilbert scheme of $X$ (see \cite{BL} for details on $G$-Hilbert schemes) and let $Y\subset \mathrm{Hilb}_G(X)$ be the irreducible component containing the free orbits. Suppose $G$ acts on $X$ such that $\omega_X$ is locally trivial in $\mathrm{Coh}_G(X)$ and write $Z\subset X\times Y$ for the universal closed subscheme. Then there is a commutative diagram of schemes
\begin{displaymath}
\begin{xy}
  \xymatrix{
      Z \ar[r]^{q} \ar[d]_{p}    &   X\ar[d]^{\pi}                   \\
      Y\ar[r]^{\tau}             &   X/\!\!/G           
  }
\end{xy}
\end{displaymath} such that $q$ and $\tau$ are birational and $p$ and $\pi$ finite. Moreover $p$ is flat. One then has the \emph{derived McKay-correspondence} (see \cite{BK}, Theorem 1.1):

\begin{thm}
Let $X$, $G$ and $Y$ be as in the diagram and suppose that $\omega_X$ is locally trivial in $\mathrm{Coh}_G(X)$. Suppose furthermore, $\mathrm{dim}(Y\times_{X/\!\!/G} Y)<\mathrm{dim}(X)+1$, then 
\begin{eqnarray*}
\mathbb{R}q_*\circ p^*\colon D^b(Y)\stackrel{\sim}\longrightarrow D^b_G(X)
\end{eqnarray*} is an equivalence and $\tau: Y\rightarrow X/\!\!/G$ a crepant resolution. 
\end{thm}
The condition that $\omega_X$ is locally trivial in $\mathrm{Coh}_G(X)$ is for instance fulfilled if $G_x\subset \mathrm{SL}(T_x)$ for all closed points $x\in X$, where $G_x$ is the stabilizer subgroup and $T_x$ the tangent space.

If $\mathrm{dim}(X)\leq 3$, the $G$-Hilbert scheme $\mathrm{Hilb}_G(X)$ is irreducible and hence there is an equivalence $\mathbb{R}q_*\circ p^*\colon D^b(\mathrm{Hilb}_G(X))\stackrel{\sim}\rightarrow D^b_G(X)$. In this case $\mathrm{Hilb}_G(X)\rightarrow X/\!\!/G$ is a crepant resolution (see \cite{BK}, Theorem 1.2). Note that Blume \cite{BL} proved the classical McKay correspondence for non algebraically closed fields of characteristic zero via Galois descent. 
\begin{exam}
\textnormal{Let $C$ be an elliptic curve over an algebraically closed field of characteristic zero with $G=\{id,-id\}=\mathrm{Aut}(C)$, i.e. $j\neq 0$ and $j\neq 1728$. Note that $C$ cannot have a tilting object, since the Grothendieck group is not a free abelian group of finite rank. Now $C/\!\!/G \simeq \mathbb{P}^1$ and hence it is its own crepant resolution. As $\mathbb{P}^1$ admits a tilting bundle, the derived McKay correspondence gives a tilting object on $[C/G]$}.  
\end{exam}
Following the idea of exploiting the derived McKay correspondence to obtain further examples of schemes having tilting objects, we state the following useful consequence of Theorem 4.2.
\begin{cor}
Let $X$, $G$ and $Y$ be as in Theorem 4.9, with $\omega_X$ being locally trivial in $\mathrm{Coh}_G(X)$. Suppose $\mathrm{dim}(Y\times_{X/\!\!/G} Y)<\mathrm{dim}(X)+1$ and that $\mathcal{T}^{\bullet}\in D_G(\mathrm{Qcoh}(X))$, considered as an object in $D(\mathrm{Qcoh}(X))$, is a tilting object on $X$. Then $Y$ admits a tilting object.
\end{cor}
\begin{proof}
Since $\mathcal{T}^{\bullet}\in D_G(\mathrm{Qcoh}(X))$, considered as an object in $D(\mathrm{Qcoh}(X))$, is a tilting object on $X$, we know from Theorem 4.2 that $\mathcal{T}^{\bullet}\otimes k[G]$ is a tilting object on $[X/G]$. As $\mathrm{dim}(Y\times_{X/\!\!/G} Y)<\mathrm{dim}(X)+1$, we have the derived McKay correspondence $\mathbb{R}q_*\circ p^*\colon D^b(Y)\stackrel{\sim}\rightarrow D^b_G(X)$. If we denote by $F$ the inverse of $\mathbb{R}q_*\circ p^*$, then $F(\mathcal{T}^{\bullet}\otimes k[G])$ is a tilting object on $Y$.
\end{proof}
Note that there is no reason for the inverse of the functor $\mathbb{R}q_*\circ p^*$ to send a coherent sheaf to a coherent sheaf. So in general the McKay correspondence gives a tilting object on $Y$ even if $[X/G]$ admits a coherent tilting sheaf. 
\begin{cor}
For $n\leq 3$, let $G\subset \mathrm{Aut}(\mathbb{P}^n)$ be a finite subgroup acting linearly on $\mathbb{P}^n$. Suppose $\omega_{\mathbb{P}^n}$ is locally trivial in $\mathrm{Coh}_G(\mathbb{P}^n)$. Then $\mathrm{Hilb}_G(\mathbb{P}^n)$ admits a tilting object.
\end{cor}
\begin{proof}
The sheaf $\bigoplus^n_{i=0}\mathcal{O}_{\mathbb{P}^n}(i)$ is a tilting bundle on $\mathbb{P}^n$ equipped with an equivariant structure (see Example 4.3). Corollary 4.11 and the discussion right after Theorem 4.9 give a tilting object on $\mathrm{Hilb}_G(\mathbb{P}^n)$.  
\end{proof}
\begin{cor}
Let $X=\mathrm{Grass}(d,n)$ be the Grassmannian of Example 4.4. Let $G$ be a finite subgroup of $\mathrm{PGL}_n(k)$ acting linearly on $X$ and suppose $\omega_X$ is locally trivial in $\mathrm{Coh}_G(X)$. Let $Y\subset \mathrm{Hilb}_G(X)$ be the irreducible component containing the free orbits and suppose $\mathrm{dim}(Y\times_{X/\!\!/G} Y)<\mathrm{dim}(X)+1$. Then $Y$ admits a tilting object. 
\end{cor}
\begin{proof}
This follows from Example 4.4 and Corollary 4.11.
\end{proof}
\section{Tilting bundles on $[\mathbb{A}(\mathcal{E})/G]$ and $[\mathbb{P}(\mathcal{E})/G]$}
We start with a generalization of results given in \cite{BR} and \cite{BS}. In loc. cit., among others, the existence of tilting objects on certain total spaces is proved. Below we study total spaces with finite group actions.\\

Let $G$ be a finite group acting on a smooth projective $k$-scheme $X$. Furthermore, let $\mathcal{E}$ be an equivariant locally free sheaf of finite rank. Consider the total space $\mathbb{A}(\mathcal{E})=\mathcal{S}pec(S^{\bullet}(\mathcal{E}))$, where $S^{\bullet}(\mathcal{E})=\mathrm{Sym}(\mathcal{E})$ is the symmetric algebra of $\mathcal{E}$. Since $\mathcal{E}$ admits an equivariant structure, the group $G$ acts on $\mathbb{A}(\mathcal{E})$ in the natural way. Note that the total space comes equipped with a $G$-morphism $\pi:\mathbb{A}(\mathcal{E})\rightarrow X$ which is affine. Assuming the existence of a tilting bundle $\mathcal{T}$ on $[X/G]$, the question arises if $[\mathbb{A}(\mathcal{E})/G]$ admits a tilting bundle, too. There is a natural candidate for a tilting bundle on $[\mathbb{A}(\mathcal{E})/G]$. 

Consider the tilting bundle $\mathcal{T}$ on $[X/G]$. Then the pullback $\pi^*\mathcal{T}$ is a locally free sheaf on $\mathbb{A}(\mathcal{E})$ with a natural equivariant structure. Below we prove that under a certain condition, $\pi^*\mathcal{T}$ is a tilting bundle on $[\mathbb{A}(\mathcal{E})/G]$. We can also show the finiteness of the global dimension of $\mathrm{End}_G(\pi^*\mathcal{T})$. Since $\mathbb{A}(\mathcal{E})$ is not projective over $k$, Proposition 3.5 cannot be applied.  
But the proof of Proposition 3.5 also works if the endomorphism algebra is required to be a noetherian ring. Since $\mathbb{A}(\mathcal{E})$ is a noetherian scheme, $G$ a finite group and $\pi^*\mathcal{T}$ a coherent sheaf, one easily verifies that $\mathrm{End}_G(\pi^*\mathcal{T})$ is indeed a noetherian ring. With this fact we now prove the following result:

\begin{thm}
Let $X$ be a smooth projective $k$-scheme, $G$ a finite group acting on $X$ and $\mathcal{E}$ an equivariant locally free sheaf of finite rank. Suppose $\mathcal{T}$ is a tilting bundle on $[X/G]$. If $H^i(X,\mathcal{T}^{\vee}\otimes \mathcal{T}\otimes S^l(\mathcal{E}))=0$ for all $i\neq 0$ and all $l>0$, then $\pi^*\mathcal{T}$ is a tilting bundle on $[\mathbb{A}(\mathcal{E})/G]$.
\end{thm} 
\begin{proof}
Note that $\pi^*\mathcal{T}$ is a coherent sheaf on $[\mathbb{A}(\mathcal{E})/G]$ and is therefore a compact object of $D(\mathrm{Qcoh}([\mathbb{A}(\mathcal{E})/G]))$ (see \cite{BV}). We now show that $\mathrm{Hom}_G(\pi^*\mathcal{T},\pi^*\mathcal{T}[i])=0$ for $i\neq 0$. Adjunction of $\pi^*$ and $\pi_*$, projection formula and Lemma 2.2 give
\begin{eqnarray*}
\mathrm{Hom}_G(\pi^*\mathcal{T},\pi^*\mathcal{T}[i])&\simeq &\mathrm{Hom}_G(\mathcal{T},\mathbb{R}\pi_*\pi^*\mathcal{T}[i])\\
 &\simeq &\mathrm{Hom}_G(\mathcal{T},S^{\bullet}(\mathcal{E})\otimes\mathcal{T}[i])\\
&\simeq &\mathrm{Hom}(\mathcal{T},S^{\bullet}(\mathcal{E})\otimes\mathcal{T}[i])^G.
\end{eqnarray*}
Now for a fixed $l>0$ one has 
\begin{center}
$\mathrm{Hom}(\mathcal{T},S^l(\mathcal{E})\otimes \mathcal{T}[i])\simeq \mathrm{Ext}^i(\mathcal{T},S^l(\mathcal{E})\otimes\mathcal{T})\simeq H^i(X,\mathcal{T}^{\vee}\otimes \mathcal{T}\otimes S^l(\mathcal{E}))$.
\end{center} By assumption, $H^i(X,\mathcal{T}^{\vee}\otimes \mathcal{T}\otimes S^l(\mathcal{E}))=0$ for all $i\neq 0$ and all $l>0$ and therefore $\mathrm{Hom}(\mathcal{T},S^{\bullet}(\mathcal{E})\otimes\mathcal{T}[i])^G=0$ for $i\neq 0$. Thus $\mathrm{Hom}_G(\pi^*\mathcal{T},\pi^*\mathcal{T}[i])=0$ for $i\neq 0$ and the Ext vanishing holds. 

It remains to prove that $\pi^*\mathcal{T}$ generates $D_G(\mathrm{Qcoh}(\mathbb{A}(\mathcal{E})))$. 
So we take an object $\mathcal{F}^{\bullet}\in D_G(\mathrm{Qcoh}(\mathbb{A}(\mathcal{E})))$ and assume $\mathbb{R}\mathrm{Hom}_G(\pi^*\mathcal{T},\mathcal{F}^{\bullet})=0$. Adjunction of $\pi^*$ and $\pi_*$ implies $\mathbb{R}\mathrm{Hom}_G(\mathcal{T},\pi_*\mathcal{F}^{\bullet})=0$. Since $\mathcal{T}$ is a tilting bundle on $[X/G]$, we get $\pi_*\mathcal{F}^{\bullet}=0$. As $\pi$ is affine, $\mathcal{F}^{\bullet}=0$ and hence $\pi^*\mathcal{T}_G$ generates $D_G(\mathrm{Qcoh}(\mathbb{A}(\mathcal{E})))$. Finally, since $\mathrm{End}_G(\pi^*\mathcal{T})$ is noetherian, the arguments in the proof of Proposition 3.5 show that the global dimension of $\mathrm{End}_G(\pi^*\mathcal{T})$ is indeed finite (notice that the noetherian property for the arguments of the proof of Proposition 3.5 is enough to conclude the finiteness of the global dimension).  
\end{proof}

If $X$ is a Fano variety with $\mathcal{E}=\omega_X$ and $G=1$ one verifies $H^i(X,\mathcal{T}^{\vee}\otimes \mathcal{T}\otimes S^l(\mathcal{E}))=0$ for all $i\neq 0$ and all $l>0$ and Theorem 5.1 gives \cite{BS}, Theorem 3.6 (see also \cite{BRI}, Proposition 4.1). For $X=\mathrm{Spec}(\mathbb{C})$, Theorem 5.1 gives \cite{BR}, Theorem 4.2.1. The arguments in the proof of Theorem 5.1 also unify and simplify the arguments given in the proves of Theorems 4.2.1 and 5.3.1 in \cite{BR}. From a representation-theoretic point of view it would also be of interest to figure out for which equivariant locally free sheaves $\mathcal{E}$ the endomorphism algebra $\mathrm{End}_G(\pi^*\mathcal{T})$ is Koszul. The existence of tilting bundles on certain total spaces also led Weyman and Zhao \cite{WZ} to a construction of non-commutative desingularizations. 
\begin{exam}
\textnormal{Let $X$ be a smooth projective $k$-scheme and $G$ a finite group acting on $X$. We take an equivariant ample invertible sheaf $\mathcal{L}$. Such a $\mathcal{L}$ always exist by the following argument: Let $\mathcal{M}$ be an ample invertible sheaf on $X$, then $g^*\mathcal{M}$ is ample for any $g\in G$. Now the tensor product $\bigotimes_{g\in G} g^*\mathcal{M}$ is ample and has a natural equivariant structure $\lambda$. Take $(\mathcal{L}, \lambda)=(\bigotimes_{g\in G} g^*\mathcal{M}, \lambda)$}. 

\textnormal{Let $\mathcal{T}$ be a tilting bundle on $[X/G]$ and set $\mathcal{E}=\mathcal{L}^{\otimes N}$. By the ampleness of $\mathcal{L}$, there exists a natural number $n\gg 0$ such that for all $N\geq n$ we have 
\begin{center}
$H^i(X,\mathcal{T}^{\vee}\otimes \mathcal{T}\otimes S^l(\mathcal{E}))\simeq H^i(X,\mathcal{T}^{\vee}\otimes \mathcal{T}\otimes \mathcal{L}^{\otimes l\cdot N})=0$
\end{center}
for all $i\neq 0$ and all $l>0$. In this case the stack $[\mathbb{A}(\mathcal{L}^{\otimes N})/G]$ admits a tilting bundle.}
\end{exam}
In view of Theorem 5.1, it is very natural to consider projective bundles with group actions. A semiorthogonal decomposition for the equivariant derived category projective bundles was constructed by Elagin \cite{EL}. Below we prove that if $[X/G]$ has a tilting bundle, then so does $[\mathbb{P}(\mathcal{E})/G]$. We start with some preliminary observation. 

Let $X$ be a smooth projective $k$-scheme and $G$ a finite group acting on $X$. Let $\mathcal{E}$ be an equivariant locally free sheaf of rank $r$ on $X$. We get a projective bundle $\mathbb{P}(\mathcal{E})$ on which $G$ acts naturally. The structure morphism $\pi:\mathbb{P}(\mathcal{E})\rightarrow X$ is a $G$-morphism and one has a semiorthogonal decomposition (see \cite{EL}, Theorem 4.3)
\begin{eqnarray}
D^b_G(\mathbb{P}(\mathcal{E}))=\langle \pi^*D^b_G(X),\pi^*D^b_G(X)\otimes \mathcal{O}_{\mathcal{E}}(1),...,\pi^*D^b_G(X)\otimes \mathcal{O}_{\mathcal{E}}(r-1)\rangle.
\end{eqnarray} 
Here $\pi^*D^b_G(X)\otimes \mathcal{O}_{\mathcal{E}}(i)$ denotes the subcategory of $D^b_G(\mathbb{P}(\mathcal{E}))$ consisting of objects of the form $\pi^*\mathcal{F}^{\bullet}\otimes \mathcal{O}_{\mathcal{E}}(i)$, where $\mathcal{F}^{\bullet}\in D^b_G(X)$. One easily proves the following lemma. 
\begin{lem}
Let $X$ be a smooth projective $k$-scheme and $G$ a finite group acting on $X$. Let $\mathcal{E}$ be an equivariant locally free sheaf of rank $r$ and $\mathbb{P}(\mathcal{E})$ the projective bundle. Let $\mathcal{A}^{\bullet}\in D^b_G(X)$ and suppose $\langle\mathcal{A}^{\bullet}\rangle=D_G^b(X)$. Then $\langle\bigoplus^{r-1}_{i=0}\pi^*\mathcal{A}^{\bullet}\otimes\mathcal{O}_{\mathcal{E}}(i)\rangle=D_G^b(\mathbb{P}(\mathcal{E}))$.
\end{lem}
\begin{proof}
First we note that $D_G^b(X)$ is derived equivalent to $\pi^*D^b_G(X)\otimes \mathcal{O}_{\mathcal{E}}(j)$ via the functor $\mathcal{F}\mapsto \pi^*\mathcal{F}\otimes\mathcal{O}_{\mathcal{E}}(j)$ (see \cite{EL}). 
Since $\langle\mathcal{A}^{\bullet}\rangle=D_G^b(X)$ and $D_G^b(X)$ is derived equivalent to $\pi^*D^b_G(X)\otimes \mathcal{O}_{\mathcal{E}}(j)$, we obtain $\langle\pi^*\mathcal{A}^{\bullet}\otimes\mathcal{O}_{\mathcal{E}}(j)\rangle=\pi^*D^b_G(X)\otimes \mathcal{O}_{\mathcal{E}}(j)$. Finally $\langle\bigoplus^{r-1}_{i=0}\pi^*\mathcal{A}^{\bullet}\otimes\mathcal{O}_{\mathcal{E}}(i)\rangle=D_G^b(\mathbb{P}(\mathcal{E}))$ in view of (2).
\end{proof}
With this lemma, we now prove the following:
\begin{thm}
Let $X$, $G$ and $\mathcal{E}$ be as in Lemma 5.3. If $[X/G]$ has a tilting bundle, then so does $[\mathbb{P}(\mathcal{E})/G]$.
\end{thm}
\begin{proof}
Let $\mathcal{T}$ be the tilting bundle on $[X/G]$ and $\pi\colon \mathbb{P}(\mathcal{E})\rightarrow X$ the projection. We consider the compact object $\mathcal{R}=\bigoplus^{r-1}_{i=0}\pi^*\mathcal{T}\otimes \mathcal{O}_{\mathcal{E}}(i)$.  
Equivariant adjuction of $\pi^*$ and $\pi_*$ and projection formula yield for $0\leq r_1,r_2\leq r-1$:
\begin{center}
$\mathrm{Hom}_G(\pi^{*}\mathcal{T}\otimes \mathcal{O}_{\mathcal{E}}(r_1),\pi^{*}\mathcal{T}\otimes \mathcal{O}_{\mathcal{E}}(r_2)[m])\simeq \mathrm{Hom}_G(\mathcal{T},\mathcal{T}\otimes \mathbb{R}\pi_{*}\mathcal{O}_{\mathcal{E}}(r_2-r_1)[m])$.
\end{center}
If $r_1=r_2$ we have $\mathbb{R}\pi_*\mathcal{O}_{\mathcal{E}}(r_2-r_1)\simeq \mathcal{O}_X$ and hence
\begin{center}
$\mathrm{Hom}_G(\pi^{*}\mathcal{T}\otimes \mathcal{O}_{\mathcal{E}}(r_1),\pi^{*}\mathcal{T}\otimes \mathcal{O}_{\mathcal{E}}(r_2)[m])\simeq \mathrm{Ext}^m_G(\mathcal{T},\mathcal{T})=0$
\end{center} for $m>0$, since $\mathcal{T}$ is a tilting bundle on $[X/G]$. If $0\leq r_2<r_1\leq r-1$, we have $r_2-r_1>-r$ and hence $\mathbb{R}\pi_*\mathcal{O}_{\mathcal{E}}(r_2-r_1)=0$ (see \cite{HA}). This gives
\begin{center}
$\mathrm{Hom}_G(\pi^{*}\mathcal{T}\otimes \mathcal{O}_{\mathcal{E}}(r_1),\pi^{*}\mathcal{T}\otimes \mathcal{O}_{\mathcal{E}}(r_2)[m])\simeq \mathrm{Ext}^m_G(\mathcal{T},0)=0$
\end{center} for all $m\geq 0$. It remains the case $0\leq r_1<r_2\leq r-1$. In this case we get for $l=r_2-r_1$, $\mathbb{R}\pi_*\mathcal{O}_{\mathcal{E}}(r_2-r_1)\simeq S^l(\mathcal{E})$ (see \cite{HA}) and therefore
\begin{eqnarray*}
\mathrm{Hom}_G(\pi^{*}\mathcal{T}\otimes \mathcal{O}_{\mathcal{E}}(r_1),\pi^{*}\mathcal{T}\otimes \mathcal{O}_{\mathcal{E}}(r_2)[m])&\simeq &\mathrm{Ext}^m_G(\mathcal{T},\mathcal{T}\otimes S^l(\mathcal{E}))\\
&\simeq& H^m(X,\mathcal{T}^{\vee}\otimes \mathcal{T}\otimes S^l(\mathcal{E}))^G.
\end{eqnarray*} 
To achieve the vanishing of the latter cohomology, we take an equivariant ample invertible sheaf $(\mathcal{L}, \lambda)$ on $X$. Such a $(\mathcal{L},\lambda)$ always exists as $X$ is projective (see Example 5.2). By the ampleness of $\mathcal{L}$, there is for a fixed $l>0$ a natural number $n_l \gg 0$ such that 
\begin{eqnarray*}
H^m(X,\mathcal{T}^{\vee}\otimes\mathcal{T}\otimes S^l(\mathcal{E}\otimes\mathcal{L}^{\otimes n_l}))&\simeq&H^m(X,\mathcal{T}^{\vee}\otimes\mathcal{T}\otimes S^l(\mathcal{E})\otimes\mathcal{L}^{\otimes n_l\cdot l})=0 
\end{eqnarray*} for $m>0$. Since $0<l\leq r-1$, we have only finitely many $l>0$ and we can choose $n>\mathrm{max}\{n_l\mid 0<l\leq r-1\}$ so that for $\mathcal{L}^{\otimes n}$ we have
\begin{eqnarray*}
H^m(X,\mathcal{T}^{\vee}\otimes\mathcal{T}\otimes S^l(\mathcal{E}\otimes\mathcal{L}^{\otimes n}))&\simeq&H^m(X,\mathcal{T}^{\vee}\otimes\mathcal{T}\otimes S^l(\mathcal{E})\otimes\mathcal{L}^{\otimes n\cdot l})= 0
\end{eqnarray*} for $m>0$ and all $0<l\leq r-1$. This implies the Ext vanishing of $\mathcal{R}':=\bigoplus^{r-1}_{i=0}\pi^*\mathcal{T}\otimes \mathcal{O}_{\mathcal{E}'}(i)$ on $\mathbb{P}(\mathcal{E}')$, where $\mathcal{E}'=\mathcal{E}\otimes \mathcal{L}^{\otimes n}$.

As $\mathbb{P}(\mathcal{E}')$ is projective and smooth, the compact objects of $D_G(\mathrm{Qcoh}(\mathbb{P}(\mathcal{E}')))$ are all of $D^b_G(\mathbb{P}(\mathcal{E}'))$ (see \cite{BR}, p.39). From Lemma 5.3 and Theorem 3.1 we conclude that $\mathcal{R}'$ generates $D_G(\mathrm{Qcoh}(\mathbb{P}(\mathcal{E}')))$. Hence $\mathcal{R}'=\bigoplus^{r-1}_{i=0}\pi^*\mathcal{T}\otimes \mathcal{O}_{\mathcal{E}'}(i)$ is a tilting bundle on $[\mathbb{P}(\mathcal{E}')/G]$. As $\mathbb{P}(\mathcal{E}')$ and $\mathbb{P}(\mathcal{E})$ are isomorphic as $G$-schemes, we obtain a tilting bundle on $[\mathbb{P}(\mathcal{E})/G]$.
\end{proof}

To apply the above theorem one has to find stacks $[X/G]$ admitting a tilting bundle. This can be done for instance with Theorem 4.1. Therefore, combining Theorem 4.1 and 5.4, we obtain further examples of quotient stacks with tilting bundles.
\begin{rema}
\textnormal{In the proof of Theorem 5.4 we used the ampleness of $\mathcal{L}$ to achieve the vanishing of $H^m(X,\mathcal{T}^{\vee}\otimes\mathcal{T}\otimes S^l(\mathcal{E}\otimes\mathcal{L}^{\otimes n}))$. For that reason it is not easy to generalize the result for the case where $[X/G]$ admits an arbitrary tilting object.} 
\end{rema}
\begin{exam}
\textnormal{Let $G\subset\mathrm{PGL}_{n+1}(k)$ be a finite subgroup acting linearly on $\mathbb{P}^n$. Example 4.2, shows that $[\mathbb{P}^n/G]$ has a tilting bundle. From Theorem 5.4 we get a tilting bundle on $[\mathbb{P}(\mathcal{E})/G]$ for any equivariant locally free sheaf $\mathcal{E}$ on $\mathbb{P}^n$.} 
\end{exam}
\begin{exam}
\textnormal{Let $G\subset \mathrm{PGL}_n(k)$ be a finite subgroup acting linearly on $X=\mathrm{Grass}_k(d,n)$. Example 4.4 shows that $[X/G]$ admits a tilting bundle. From Theorem 5.4 we get a tilting bundle on $[\mathbb{P}(\mathcal{E})/G]$ for any equivariant locally free sheaf $\mathcal{E}$ on $X$.}  
\end{exam}
\begin{exam}
\textnormal{Let $X$ be a Brauer--Severi variety over $k$ and $G\subset\mathrm{Aut}(X)$ a finite subgroup acting on $X$ as described in Example 4.7. Then there is a tilting bundle on $[X/G]$. From Theorem 5.4 we get a tilting bundle on $[\mathbb{P}(\mathcal{E})/G]$ for any equivariant locally free sheaf $\mathcal{E}$ on $X$.}
\end{exam}
\section{Application: Orlov's dimension conjecture}
As an application of the results of the previews sections we provide some further evidence for a conjecture on the Rouquier dimension of derived categories formulated by Orlov \cite{O1}.\\

Let $\mathcal{D}$ be a triangulated category. For two full triangulated subcategories $\mathcal{M}$ and $\mathcal{N}$ of $\mathcal{D}$ we want to denote by $\mathcal{M}\star \mathcal{N}$ the full subcategory consisting of objects $R$ such that there exists a distinguished triangles of the form
\begin{center}
$X_1\longrightarrow R\longrightarrow X_2\longrightarrow X_1[1]$,
\end{center}
where $X_1\in \mathcal{M}$ and $X_2\in \mathcal{N}$. Then set $\mathcal{M}\diamond \mathcal{N}=\langle \mathcal{M}\star \mathcal{N}\rangle$. We inductively define $\langle \mathcal{M}\rangle_i=\langle \mathcal{M}\rangle_{i-1}\diamond \langle \mathcal{M}\rangle$ and set $\langle \mathcal{M}\rangle_1$ to be $\langle \mathcal{M}\rangle$.
\begin{defi}
\textnormal{The dimension of a triangulated category $\mathcal{D}$, denoted by $\mathrm{dim}(\mathcal{D})$, is the smallest integer $n\geq 0$ such that there exists an object $A$ for which $\langle A\rangle_{n+1}= \mathcal{D}$. We define the dimension to be $\infty$ if there is no such $A$.}
\end{defi}
There is a lower and a upper bound for the dimension of the bounded derived category of coherent sheaves of a scheme $X$. Rouquier \cite{RO}, who originally introduced the notion of dimension of triangulated categories, proved that for reduced separated schemes $X$ of finite type over $k$ a lower bound is given by $\mathrm{dim}(D^b(X))\geq\mathrm{dim}(X)$ (see \cite{RO}, Proposition 7.17), whereas for smooth quasiprojective $k$-schemes $X$ an upper bound is given by $\mathrm{dim}(D^b(X))\leq 2\mathrm{dim}(X)$ (see \cite{RO}, Proposition 7.9). There is the following conjecture:
\begin{con2}(\cite{O1})
\textnormal{If $X$ is a smooth integral and separated scheme of finite type over $k$, then $\mathrm{dim}(D^b(X))=\mathrm{dim}(X)$}.
\end{con2} In loc. cit. it is proved that the conjecture holds for smooth projective curves $C$ of genus $g\geq 1$. For curves of genus $g=0$ this is an easy observation and well known. Therefore $\mathrm{dim} (D^b(C))=1$ for all smooth projective curves $C$. Additionally, the conjecture is known to be true in the following cases: 
\begin{itemize}
  \item affine schemes of finite type over $k$, certain flags and quadrics \cite{RO}.
   \item del Pezzo surfaces, certain Fano three-folds, Hirzebruch surfaces, toric surfaces with nef anti-canonical divisor and certain toric Deligne--Mumford stacks over $\mathbb{C}$ \cite{BF}.
	\end{itemize} 
Ballard and Favero \cite{BF} extended the above conjecture to certain Deligne--Mumford stacks. Hereafter, we denote by $D^b(\mathcal{X})$ the bounded derived category of the abelian category of coherent sheaves on a separated Deligne--Mumford stack $\mathcal{X}$ of finite type over $k$. For details on Deligne--Mumford stacks we refer to \cite{DM} and to the appendix of \cite{V}.
\begin{con2}(\cite{BF})
\textnormal{Let $\mathcal{X}$ be a smooth and tame Deligne--Mumford stack of finite type over $k$ with quasiprojective coarse moduli space, then $\mathrm{dim}(D^b(\mathcal{X}))=\mathrm{dim}(\mathcal{X})$.}
\end{con2}
Among others, in loc. cit. it is proved the following theorem (see \cite{BF}, Theorem 3.2).
\begin{thm}
Let $\mathcal{X}$ be a smooth, proper, tame and connected Deligne--Mumford stack with projective coarse moduli space. Suppose $\langle\mathcal{T}\rangle=D^b(\mathcal{X})$ satisfying $\mathrm{Hom}_{D^b(\mathcal{X})}(\mathcal{T},\mathcal{T}[i])=0$ for $i\neq 0$ and let $i_0$ be the largest $i$ for which $\mathrm{Hom}_{D^b(\mathcal{X})}(\mathcal{T},\mathcal{T}\otimes \omega^{\vee}_{\mathcal{X}}[i])$ is non-zero.
If $k$ is a perfect field, then $\mathrm{dim} (D^b(\mathcal{X}))=\mathrm{dim}(\mathcal{X})+i_0$. 
\end{thm}

We now want to apply Theorem 6.2 and results of the previews sections to produce some more examples where the above conjecture holds true.
\begin{prop}
Let $k$ be a perfect field and $X$ and $G$ as in Theorem 4.1. Suppose $X$ is connected and $\mathcal{T}$ is a coherent tilting sheaf on $[X/G]$. If $\mathrm{Hom}(\mathcal{T},\mathcal{T}\otimes \omega^{\vee}_X[i])=0$ for $i>0$ on $X$, then $\mathrm{dim}([X/G])=\mathrm{dim}(D^b([X/G]))$.
\end{prop}
\begin{proof}
It is easy to check that $[X/G]$ is a smooth, proper, tame and connected Deligne--Mumford stack with projective coarse moduli space $X/\!\!/G$ (see \cite{AOV} and \cite{V}). Theorem 6.2 shows that 
we have to verify $\mathrm{Hom}_G(\mathcal{T},\mathcal{T}\otimes \omega^{\vee}_{[X/G]}[i])=0$ for $i>0$.
Here $\omega_{[X/G]}$ is the dualizing object for the stack $[X/G]$. Note that the underlying sheaf of $\omega_{[X/G]}$, considered as an object in $\mathrm{Coh}_G(X)$, is $\omega_X$ (the dualizing sheaf of $X$) with a certain equivariant structure $\lambda$. Note that $\mathcal{T}\in D^b_G(X)$. From Lemma 2.2 we get 
 \begin{center}
$\mathrm{Hom}_G(\mathcal{T},\mathcal{T}\otimes \omega^{\vee}_{[X/G]}[i])\simeq \mathrm{Hom}(\mathcal{T},\mathcal{T}\otimes \omega^{\vee}_{X}[i])^G$.
\end{center} By assumption, $\mathrm{Hom}(\mathcal{T},\mathcal{T}\otimes \omega^{\vee}_X[i])=0$ for $i>0$ and hence $\mathrm{Hom}_G(\mathcal{T},\mathcal{T}\otimes \omega^{\vee}_{[X/G]}[i])=0$ 
 for $i>0$. Theorem 6.2 yields $\mathrm{dim}([X/G])=\mathrm{dim}(D^b([X/G]))$.
\end{proof}

\begin{cor}
Let $k$ be a perfect field, $G$ a finite group acting on a smooth projective $k$-scheme $X$ and $\mathcal{E}$ an equivariant locally free sheaf of rank $r$. If $[X/G]$ admits a tilting bundle, then $\mathrm{dim}([\mathbb{P}(\mathcal{E})/G])=\mathrm{dim}(D^b([\mathbb{P}(\mathcal{E})/G]))$.
\end{cor}
\begin{proof}
Denote by $\mathcal{T}$ the tilting bundle on $[X/G]$ and by $\pi\colon \mathbb{P}(\mathcal{E})\rightarrow X$ the projection. The proof of Theorem 5.4 shows that there exists an equivariant ample sheaf $(\mathcal{L},\lambda)$ such that $\mathcal{R}=\bigoplus^{r-1}_{i=0}\pi^*\mathcal{T}\otimes \mathcal{O}_{\mathcal{E}'}(i)$ is a tilting bundle on $[\mathbb{P}(\mathcal{E}')/G]$, where $\mathcal{E}'=\mathcal{E}\otimes \mathcal{L}$. As $\omega^{\vee}_{\mathcal{E}'}=\mathcal{O}_{\mathcal{E}'}(r)$, we have to verify that $\mathrm{Ext}^l(\mathcal{R},\mathcal{R}\otimes \mathcal{O}_{\mathcal{E}'}(r))=0$ for $l>0$. For $0\leq r_1,r_2\leq r-1$ we have
\begin{center}
$\mathrm{Ext}^l(\pi^{*}\mathcal{T}\otimes \mathcal{O}_{\mathcal{E}'}(r_1),\pi^{*}\mathcal{T}\otimes \mathcal{O}_{\mathcal{E}'}(r_2)\otimes \mathcal{O}_{\mathcal{E}'}(r))\simeq \mathrm{Ext}^l(\mathcal{T},\mathcal{T}\otimes \mathbb{R}\pi_{*}\mathcal{O}_{\mathcal{E}'}(r+r_2-r_1))$.
\end{center} Since $m:=r+r_2-r_1>-r$, we get $\mathbb{R}\pi_{*}\mathcal{O}_{\mathcal{E}'}(m)=S^m(\mathcal{E}')\simeq S^m(\mathcal{E})\otimes \mathcal{L}^{\otimes m}$ (see \cite{HA}). As there are only finitely many $m$, we can choose the equivariant ample sheaf $(\mathcal{L},\lambda)$ such that $H^l(X,\mathcal{T}^{\vee}\otimes \mathcal{T}\otimes S^m(\mathcal{E})\otimes \mathcal{L}^{\otimes m})=0$ for $l>0$. This gives $\mathrm{Ext}^l(\mathcal{R},\mathcal{R}\otimes \mathcal{O}_{\mathcal{E}'}(r))=0$ for $l>0$. Proposition 6.3 gives $\mathrm{dim}([\mathbb{P}(\mathcal{E}')/G])=\mathrm{dim}(D^b([\mathbb{P}(\mathcal{E}')/G]))$. As $\mathbb{P}(\mathcal{E}')$ and $\mathbb{P}(\mathcal{E})$ are isomorphic as $G$-schemes, we obtain $\mathrm{dim}([\mathbb{P}(\mathcal{E})/G])=\mathrm{dim}(D^b([\mathbb{P}(\mathcal{E})/G]))$.
\end{proof}
\begin{cor}
Let $X$ be a $n$-dimensional Brauer--Severi variety over a perfect field $k$ corresponding to a central simple algebra $A$ and $G\subset \mathrm{Aut}(X)=A^*/k^{\times}$ a finite subgroup such that the action lifts to an action of $A^*$. Then $\mathrm{dim}(D^b([X/G]))=\mathrm{dim}([X/G])$.
\end{cor}

\begin{proof}
Denote by $\mathcal{V}$ the tautological sheaf on $X$ and let $\mathcal{T}=\bigoplus^n_{i=0}\mathcal{V}^{\otimes i}$. Note that $X\otimes_k \bar{k}\simeq\mathbb{P}^n$. We know from Example 4.7 that $\mathcal{T}\otimes k[G]$ is a tilting bundle on $[X/G]$. According to Proposition 6.3 we have to verify $\mathrm{Ext}^l(\mathcal{T}\otimes k[G],\mathcal{T}\otimes k[G]\otimes \omega^{\vee}_X)=0$ for $l>0$. Notice that $\mathcal{V}^{\otimes i}\otimes_k \bar{k}\simeq \mathcal{O}_{\mathbb{P}^n}(-i)^{\oplus (n+1)^i}$ and $\omega_X=\mathcal{O}_X(-n-1)$. One easily verifies $\mathrm{Ext}^l(\mathcal{T},\mathcal{T}\otimes \omega^{\vee}_X)\otimes_k \bar{k}=0$ for $l>0$ on $\mathbb{P}^n$. Thus $\mathrm{Ext}^l(\mathcal{T},\mathcal{T}\otimes \omega^{\vee}_X)=0$ for $l>0$ on $X$. As $k[G]=\bigoplus_j W^{\oplus\mathrm{dim}(W_j)}_j$, it is enough to consider
\begin{center}
$\mathrm{Ext}^l(\mathcal{T}\otimes W_r,\mathcal{T}\otimes W_s\otimes\omega^{\vee}_X)\simeq \mathrm{Ext}^l(\mathcal{T},\mathcal{T}\otimes\omega^{\vee}_X)\otimes \mathrm{Hom}(W_r,W_s)$.
\end{center} Now $\mathrm{Ext}^l(\mathcal{T},\mathcal{T}\otimes \omega^{\vee}_X)=0$ for $l>0$ implies $\mathrm{Ext}^l(\mathcal{T}\otimes k[G],\mathcal{T}\otimes k[G]\otimes \omega^{\vee}_X)=0$ for $l>0$.
\end{proof}

\begin{cor}
Let $k$ be an algebraically closed field of characteristic zero and $G$ a finite subgroup of $\mathrm{PGL}_n(k)$ acting linearly on $X=\mathrm{Grass}(d,n)$. Then $\mathrm{dim}([X/G])=\mathrm{dim} (D^b([X/G]))$.
\end{cor}
\begin{proof}
Let $\mathcal{T}=\bigoplus_{\lambda}\Sigma^{\lambda}(\mathcal{S})$, where $\mathcal{S}$ is the tautological sheaf on $X$ and $\Sigma^{\lambda}$ the Schur functor (see \cite{KA}). From Example 4.4 and Theorem 4.2 we know that $\mathcal{T}\otimes k[G]$ is a tilting bundle on $[X/G]$. Note that $\omega_X=\mathcal{O}_X(-n)$. To apply Proposition 6.3 we have to verify $\mathrm{Ext}^i(\mathcal{T}\otimes k[G],\mathcal{T}\otimes k[G]\otimes \mathcal{O}_X(n))=0$ for $i>0$. By the isomorphism 
\begin{center}
$\mathrm{Ext}^i(\mathcal{T}\otimes W_r,\mathcal{T}\otimes W_s\otimes \mathcal{O}_X(n) )\simeq \mathrm{Ext}^i(\mathcal{T},\mathcal{T}\otimes \mathcal{O}_X(n))\otimes \mathrm{Hom}(W_r,W_s)$
\end{center}
it is enough to show $\mathrm{Ext}^i(\Sigma^{\lambda}(\mathcal{S}),\Sigma^{\mu}(\mathcal{S})\otimes \mathcal{O}_X(n))\simeq H^i(X,\Sigma^{\lambda}(\mathcal{S}^{\vee})\otimes\Sigma^{\mu}(\mathcal{S})\otimes \mathcal{O}_X(n))=0$ for $i>0$. 

It follows from the Littlewood--Richardson rule that for each irreducible summand $\Sigma^{\gamma}(\mathcal{S})\subset \mathcal{H}om(\Sigma^{\lambda}(\mathcal{S}),\Sigma^{\mu}(\mathcal{S}))\simeq \Sigma^{\lambda}(\mathcal{S}^{\vee})\otimes\Sigma^{\mu}(\mathcal{S})$, $\gamma=(\lambda_1,...,\lambda_d)$ satisfies $\gamma_1\geq \gamma_2\geq...\geq\gamma_d\geq -(n-d)$ (see \cite{KA2}, 3.3). So we can restrict ourselves to show
\begin{center}
$H^i(X,\Sigma^{\gamma}(\mathcal{S})\otimes\mathcal{O}_X(n))=0$ for $i>0$. 
\end{center} Since $\Sigma^{\gamma}(\mathcal{S})\otimes \mathcal{O}_X(n)\simeq \Sigma^{\gamma+n}(\mathcal{S})\simeq \Sigma^{-\gamma-n}(\mathcal{S}^{\vee})$, where $\gamma+n=(\gamma_1+n,...,\gamma_d+n)$, we have $\gamma_1+n\geq \gamma_2+n\geq...\geq\gamma_d+n\geq d$. The calculation of the cohomology of $\Sigma^{\alpha}(\mathcal{S}^{\vee})$ (see \cite{KA}, Lemma 2.2 or \cite{KA2}, Lemma 3.2) gives $H^i(X,\Sigma^{\gamma+n}(\mathcal{S}))=0$ for $i>0$. 
Proposition 6.3 then implies $\mathrm{dim}([X/G])=\mathrm{dim}( D^b([X/G]))$. 
\end{proof}

\begin{prop}
Let $k$, $X$, $G$ and $Y\subset \mathrm{Hilb}_G(X)$ be as in Corollary 4.11 ($\omega_X$ is supposed to be locally trivial in $\mathrm{Coh}_G(X)$). Assume $\mathrm{dim}(Y\times_{X/\!\!/G} Y)<\mathrm{dim}(X)+1$. If $[X/G]$ has a coherent tilting sheaf $\mathcal{T}$ satisfying $\mathrm{Ext}^i(\mathcal{T},\mathcal{T}\otimes \omega^{\vee}_X)=0$ for $i>0$, then $\mathrm{dim}(Y)=\mathrm{dim}(D^b(Y))$.
\end{prop}
\begin{proof}
By assumption, $k$ is of characteristic zero and hence perfect. Proposition 6.3 gives $\mathrm{dim}([X/G])=\mathrm{dim}(D^b_G(X))$. Since $\mathrm{dim}(Y)=\mathrm{dim}(X/\!\!/G)=\mathrm{dim}(X)=\mathrm{dim}([X/G])=\mathrm{dim}(D^b_G(X))$, the McKay equivalence $D^b(Y)\stackrel{\sim}\rightarrow D^b_G(X)$ yields $\mathrm{dim}(Y)=\mathrm{dim}(D^b(Y))$. 
\end{proof}
Summarizing the above observations we obtain the following theorem.
\begin{thm}
The dimension conjecture holds for:
\begin{itemize}
  \item[\bf(i)] quotient stacks $[\mathbb{P}(\mathcal{E})/G]$ as in Theorem 5.4, provided $[X/G]$ has a tilting bundle and $k$ is perfect.
	\item[\bf(ii)] quotient stacks $[X/G]$ over a perfect field $k$, where $X$ is a Brauer--Severi variety corresponding to a central simple algebra $A$ and $G\subset \mathrm{Aut}(X)=A^*/k^{\times}$ a finite subgroup such that the action lifts to an action of $A^*$.  
	\item[\bf(iii)] quotient stacks $[\mathrm{Grass}(d,n)/G]$ over an algebraically closed field $k$ of characteristic zero, provided $G\subset\mathrm{PGL}_n(k)$ is a finite subgroup acting linearly on $\mathrm{Grass}(d,n)$.
\item[\bf (iv)] $G$-Hilbert schemes $\mathrm{Hilb}_G(\mathbb{P}^n)$ over an algebraically closed field $k$ of characteristic zero, provided $n\leq 3$ and $G\subset \mathrm{PGL}_{n+1}(k)$ is a finite subgroup acting linearly on $\mathbb{P}^n$ and $\omega_{\mathbb{P}^n}$ is locally trivial in $\mathrm{Coh}_G(\mathbb{P}^n)$.
	\end{itemize} 
\end{thm}

{\small MATHEMATISCHES INSTITUT, HEINRICH--HEINE--UNIVERSIT\"AT 40225 D\"USSELDORF, GERMANY}\\
E-mail adress: novakovic@math.uni-duesseldorf.de

\end{document}